\documentclass[reqno,11pt]{amsart}
\usepackage{amsmath, latexsym, amsfonts, amssymb, amsthm, amscd}
\usepackage{graphics,epsf,psfrag}

\numberwithin{equation}{section}

\newtheorem{theorem}{Theorem}[section]
\newtheorem{lemma}[theorem]{Lemma}
\newtheorem{proposition}[theorem]{Proposition}
\newtheorem{cor}[theorem]{Corollary}
\newtheorem{rem}[theorem]{Remark}

\newcommand{\ind}{\mathbf{1}}

\newcommand{\R}{\mathbb{R}}
\newcommand{\Z}{\mathbb{Z}}
\newcommand{\N}{\mathbb{N}}
\renewcommand{\tilde}{\widetilde}

\newcommand{\bP}{{\ensuremath{\mathbf P}} }
\newcommand{\bE}{{\ensuremath{\mathbf E}} }

\DeclareMathSymbol{\leqslant}{\mathalpha}{AMSa}{"36} 
\DeclareMathSymbol{\geqslant}{\mathalpha}{AMSa}{"3E} 
\DeclareMathSymbol{\eset}{\mathalpha}{AMSb}{"3F}     
\newcommand{\dd}{\,\text{\rm d}}             

\newcommand{\bbE}{{\ensuremath{\mathbb E}} }

\newcommand{\bbL}{{\ensuremath{\mathbb L}} }

\newcommand{\bbP}{{\ensuremath{\mathbb P}} }

\newcommand{\bbR}{{\ensuremath{\mathbb R}} }

\newcommand{\ga}{\alpha}
\newcommand{\gb}{\beta}
\newcommand{\gga}{\gamma}            
\newcommand{\gd}{\delta}
\newcommand{\gep}{\varepsilon}       

\newcommand{\go}{\omega}

\newcommand{\gl}{\lambda}

\makeatletter
\def\captionfont@{\footnotesize}
\def\captionheadfont@{\scshape}

\long\def\@makecaption#1#2{%
  \vspace{2mm}
  \setbox\@tempboxa\vbox{\color@setgroup
    \advance\hsize-6pc\noindent
    \captionfont@\captionheadfont@#1\@xp\@ifnotempty\@xp
        {\@cdr#2\@nil}{.\captionfont@\upshape\enspace#2}%
    \unskip\kern-6pc\par
    \global\setbox\@ne\lastbox\color@endgroup}%
  \ifhbox\@ne 
    \setbox\@ne\hbox{\unhbox\@ne\unskip\unskip\unpenalty\unkern}%
  \fi
  \ifdim\wd\@tempboxa=\z@ 
    \setbox\@ne\hbox to\columnwidth{\hss\kern-6pc\box\@ne\hss}%
  \else 
    \setbox\@ne\vbox{\unvbox\@tempboxa\parskip\z@skip
        \noindent\unhbox\@ne\advance\hsize-6pc\par}%
\fi
  \ifnum\@tempcnta<64 
    \addvspace\abovecaptionskip
    \moveright 3pc\box\@ne
  \else 
    \moveright 3pc\box\@ne
    \nobreak
    \vskip\belowcaptionskip
  \fi
\relax
}
\makeatother
\def\writefig#1 #2 #3 {\rlap{\kern #1 truecm
\raise #2 truecm \hbox{#3}}}

\newcommand{\tf}{\textsc{f}}

\title{The Martingale approach to disorder irrelevance for pinning models}
\author{Hubert Lacoin}
\address{
Università degli Studi “Roma Tre”, Largo San Leonardo Murialdo
00146 Roma, Italia}

\email{lacoin\@@math.jussieu.fr}
\begin{document}
 \maketitle

\begin{abstract}
 This paper presents a very simple and self-contained proof of disorder irrelevance for inhomogeneous pinning models with return exponent $\alpha\in (0,1/2)$. We also give a new upper bound for the contact fraction of the disordered model at criticality.
\\
2000 \textit{Mathematics Subject Classification: 82B44, 60K37, 60K05
  }\\
  \textit{Keywords: Pinning/Wetting Models, Disordered Models, Harris Criterion, Relevant
    Disorder, Renewal Theory}
\end{abstract}

\section{Introduction and presentation of the main result}

Pinning/wetting models with quenched disorder describe the random interaction between a directed polymer and a one-dimensional defect line.
In absence of interaction, the polymer spatial configuration is modeled by $(n,S_n)_{n\ge 0}$,  where $(S_n)_{n\ge 0}$ is a Markov Chain (law $\bP$) in a certain state space $\Sigma$ (e.g.\ a simple symmetric random walk in $\Z^d$ for the $d+1$ dimensional polymer), and the initial condition is some fixed element of $\Sigma$ which by convention we call $0$.
The defect line, on the other hand, is just $\{0\}\times \Z_+$. The polymer/line interaction is the following : each time the polymer touches the line (i.e., each time $S_n=0$) it gets an energy reward/penalty which can be either positive or negative.

\medskip

The interaction with the defect line only changes the law of the return time to zero of $(S_n)_{n\ge0}$, but does not change the law of the excursions conditionally to the time of the set of visits to zero. For this reason, we focus on $\tau\subset\N$ the set of times where $S_n=0$ (this is the pinning configuration) and forget about the original Markov chain.
Alternatively, we may consider $\tau=(\tau_n)_{n\ge 0}$ as an increasing sequence starting from zero that may contain only finitely many terms  (if $n<\infty$ is the number of element of $\tau$, we write by convention  $\tau_{k}=\infty$ for $k>n$).

\medskip

Under $\bP$, $\tau$ is a renewal sequence, {\it i.e.}, $\tau_0=0$ and the variables $\tau_{n+1}-\tau_n$ (conditioned to $\tau_n<\infty$) are i.i.d.\ distributed. 

\medskip

The most interesting cases for pinning problems are the case where the law of the inter-arrival times to $0$ of the Markov chain have power-law decay, more precisely

\begin{equation}\label{regvar}
 \bP(\tau_1=n)=\frac{L(n)}{n^{1+\alpha}},
\end{equation}
with $\alpha>0$ and $L$ a slowly varying function, {\it i.e.}\ a measurable function from $(0,\infty)$ to $(0,\infty)$ such that $\lim_{x\to \infty} L(xu)/L(x)=1$ for all $u>0$ (see \cite{cf:RegVar} for more informations on slowly varying functions). We keep this assumption throughout all the paper.

\medskip
 
Now we are ready to define our model in a simple manner: given $(\go_n)_{n\in \N}$, a typical realization of a sequence of i.i.d.\ centered random variables with unit variance which have exponential moments (see \eqref{expmom}, let $\bbP$ denote the associated law), $h\in \R$ and $\gb\ge 0$, we consider the sequence of measure $\bP^{\gb,h,\go}_N$ on $\tau$ defined by
\begin{equation}
\frac{\dd \bP^{\gb,h,\go}_N}{\dd \bP}(\tau):=\frac{1}{Z_N^{\gb,h,\go}}\exp\left(\sum_{n=1}^N (h+\gb \go_n)\ind_{\{n\in \tau\}}\right),
\end{equation}
where 
\begin{equation}
 Z_N^{\gb,h,\go}:= \bE\left[\exp\left(\sum_{n=1}^N (h+\gb \go_n)\ind_{\{n\in \tau\}}\right)\right],
\end{equation}
is the renormalization factor (partition function) that makes $\bP^{\gb,h,\go}_N$ a probability law.
The term $h+\gb\go_n$ corresponds to the energy reward/penalty for a return to zero at step $n$.
We want to understand the typical behavior of $\tau$ under the measure $\bP^{\gb,h,\go}_N$ for large $N$. To that purpose,
a key quantity is the {\sl quenched} free energy of the system

\begin{equation}
\tf^q(\gb,h):=\lim_{N\to \infty}\frac{1}{N}\bbE \log Z_N^{\gb,h,\go}=\lim_{N\to \infty} \frac{1}{N} \log Z_N^{\gb,h,\go},
\end{equation}
where the existence of the limit and the second inequality hold $\bbP-a.s.$
The existence of these limits follows from some superadditivity properties of the system (see e.g.\ \cite[Chapter 4]{cf:Book}). The function $h\mapsto \tf^q(\gb,h)$ is non-negative, convex, and non-decreasing. A phase transition in $h$ occurs in the system at the value

\begin{equation}
 h_c(\gb):=\inf\left\{h\ : \tf^q(\gb, h)>0\right\},
\end{equation}
which we refer to as the quenched critical point.

It is known that $h<h_c(\gb)$ corresponds to the delocalized phase, where there is at most $O(\log (N))$ contact with the defect line before $N$ with large probability, whereas $h>h_c(\gb)$ corresponds to the localized phase where the number point in $\tau\cap [0,N]$ under $\bP_N^{\gb,\go,h}$ is of order $N$ (We refer to \cite[Chapters 7,8]{cf:Book} for further literature and discussion of this point).
We write $\tf(h)$ for $\tf^q(0,h)$ and $h_c$ for $h_c(0)$.

In analogy with the quenched free energy, the {\sl annealed} free energy is defined by
\begin{equation}
 \tf^a(\gb,h):=\lim_{N\to \infty}\frac{1}{N}\log  \bbE Z_N^{\gb,h,\go}=\tf(h+\gl(\gb)).
\end{equation}
(the second equality is obtained by using Fubini's theorem) where
\begin{equation}\label{expmom}
 \gl(\gb):=\log \bbE[\exp(\gb\go_1)]<\infty.
\end{equation}
We define also the annealed critical point for the free energy by
\begin{equation}
 h^a_c(\gb):=\inf\{ h : \  \tf^a(h,\gb)>0\}=h_c-\gl(\gb).
\end{equation}
By Jensen inequality, the annealed free-energy dominates the annealed one. Indeed
\begin{equation}
\bbE \log Z_N^{\gb,h,\go}\le  \log \bbE Z_N^{\gb,h,\go}.
\end{equation} 
So that
\begin{equation}\label{dominates}
\begin{split}
 \tf^q(\gb,h)&\le \tf^a(\gb,h)\\
h_c(\gb)&\ge h_c^a(\gb).
\end{split}
\end{equation}

\medskip

The behavior of the polymer measure for $\gb=0$ (the homogeneous pinning model) is very well understood (see \cite{cf:Fisher, cf:Book}). In this case, the model possesses the property of being exactly solvable: one has an explicit formula for the free energy (and therefore, also for the annealed free energy for all $\gb$).
In particular one has $h_c=-\log \bP[\tau_1<\infty]$ ($h_c=0$ when $\tau$ is recurrent) and
\begin{equation}\label{exponent}
 \lim_{h\to (h_c)_+} \frac{\log \tf(h)}{\log(h-h_c)}=1\vee \alpha^{-1}.
\end{equation}
The quantity $1\vee \alpha^{-1}$ is the critical exponent for the annealed free energy.

\medskip

The Harris criterion (formulated by A.B.\ Harris \cite{cf:Harris}), predicts that for disordered systems, whether the quenched and annealed systems have the same critical behavior at high temperature (i.e. for low $\gb$) depends on the critical exponent of the annealed free energy. More precisely: it says that 
disorder is relevant (for all $\gb$) if the critical exponent is smaller than $2$ and irrelevant (for small values of $\gb$) if the critical exponent is larger than $2$. In view of \eqref{exponent}, this corresponds, for our pinning model to $\alpha>1/2$ and $\alpha<1/2$ respectively. 

 While a priori, the Harris approach yields a prediction only on critical exponents,
when specialized to pinning models it yields the stronger prediction  (see in particular \cite{cf:DHV, cf:FLNO}) that
if  $\alpha <1/2$, then both the annealed critical exponent of the free energy coincides with the quenched critical exponent
and that the two critical points (annealed and quenched) coincide. 
On the other hand, if $\ga >1/2$ quenched and annealed free energy exponent and critical point are expected to differ.

Various mathematical confirmations have been given for the validity of Harris criterion for pinning models (see \cite{cf:GT_cmp, cf:Ken, cf:T_cmp, cf:DGLT, cf:AZ}), and recently, the marginal case $\alpha=1/2$ for which Harris criterion gives no prediction has been solved \cite{cf:GLT_marg}.

In this note, we present a simple martingale method that proves the validity of Harris criterion in the case $\alpha<1/2$. Stronger versions of this result have been proved by Alexander \cite{cf:Ken} for Gaussian environment, an alternative approach was found later by Toninelli \cite{cf:T_cmp}, but our new method considerably simplifies the proof and does not need any assumption on the environment (whereas both other papers focused on the Gaussian case).
We do not cover the special case $\alpha=0$ which was treated by Alexander and Zygouras \cite{cf:AZ_new}. Our method allows us also to derive new results about the property of the trajectories at the critical point $h_c(\gb)$. We present now our main result.

\begin{theorem}\label{mainrel}
If $\alpha\in(0,1/2)$ or if $\alpha=1/2$ and $L$ is such that
\begin{equation}\label{convers}
 \sum_{n=1}^\infty \frac{1}{n L(n)^2}<\infty,
\end{equation}
there exists $\gb_2>0$ such that for all $\gb\le \gb_2$, $h_c(\gb)=h^a_c(\gb)$, and 
\begin{equation}
 \lim_{h\to h^a_c(\gb)^+} \frac{\log\tf^q(\gb,h)}{\log (h-h^a_c(\gb))}=\alpha^{-1}.
\end{equation}
\end{theorem}

\begin{rem}\rm
 Note that it suffices to prove $\limsup_{h\to h^a_c(\gb)^+} \frac{\log(\tf(\gb,h))}{\log(h-h^a_c(\gb))}\le \alpha^{-1}$, which implies $h_c(\gb)\ge h^a_c(\gb)$. The rest of the statement is implied by \eqref{dominates} and \eqref{exponent}.
\end{rem}

We give an explicit lower-bound for $\gb_2$ in Proposition \ref{martinG}.
Theorem \ref{mainrel} follows from Propositions \ref{expos}, \ref{technos} and \ref{martinG} that we prove in the next section. The first proposition links the expected number of contacts at the critical point and the critical exponent for the free energy.

\begin{proposition}\label{expos}
Consider $\gamma>0$.
If $h_0$ and $\gb$ are such that $\tf(\gb,h_0)=0$ and if there exists $c>0$ such that
\begin{equation}\label{condinf}
 \liminf_{N\to \infty}\bbP\left[  \bP^{\gb,h_0,\go}_N\left(|\tau\cap[0,N]|>N^{\gamma}\right) >c \right] >c,
\end{equation}
then $h_c(\gb)=h_0$ and for any $\theta>\gga^{-1}$
\begin{equation}
 \liminf_{h\to h_{0}^+} \tf(\gb,h)(h-h_0)^{-\theta}=\infty.
\end{equation}
In the same way, if there exists $c>0$ such that
\begin{equation}\label{chemoss}
 \limsup_{N\to \infty}\bbP\left[  \bP^{\gb,h_0,\go}_N\left(|\tau\cap[0,N]|>N^{\gamma}\right) >c \right] >c
\end{equation}
then $h_c(\gb)=h_0$ and for any $\theta>\gga^{-1}$
\begin{equation}\label{chemos}
\limsup_{h\to h_{0}^+}\tf^q(\gb,h)(h-h_0)^{-\theta}=\infty.
\end{equation}
\end{proposition}

This result (the $\limsup$ version) combined with a following result of Giacomin and Toninelli \cite[Theorem 2.1]{cf:GT_cmp} gives the following consideration on polymer measure at the critical point.

\begin{cor}
 If one of the following conditions is satisfied 
\begin{itemize}
 \item [(i)] The law of $\go_1$ has bounded support.
 \item [(ii)] The law of $\go_1$ has density $d(\cdot)$ with respect to Lebesgue measure and there exists $R$ such that 
\begin{equation}
 \int_{\bbR}d(x+y)\log \left(\frac{d(x+y)}{d(y)}\right) \dd  y \le Rx^2.
\end{equation}

\end{itemize}
Then for any $\gamma>1/2$, and $\gb>0$.
 \begin{equation}\label{chemo}
\bP^{\gb,h_c(\gb),\go}_N\left(|\tau\cap[0,N]|>N^{\gamma}\right)\to 0
\end{equation}
in $\bbP$ probability.
\end{cor}

\begin{proof}
 If \eqref{chemo} does not hold, then \eqref{chemoss} and therefore \eqref{chemos}, hold for some $c>0$ and $\theta<2$. But this contradicts the conclusion of  \cite[Theorem 2.1]{cf:GT_cmp}, that says that

\begin{equation}
\limsup_{h\to h_{c}(\gb)^+} \tf^q(\gb,h)(h-h_c(\gb))^{-2}< \infty.
\end{equation}

\end{proof}

Before presenting the next result, we need some definitions.
For the techniques we are to use, we need the assumption that $\tau$ is recurrent. However, if $\tau$ is not recurrent, one can consider the system based on the renewal $\tilde \tau$ defined by $\bP(\tilde \tau_1=n)=K(n)/(\sum_{n=1}^{\infty} K(n))$ which is recurrent, and whose free energy curve is just a shift along the $h$ coordinate of the free energy curve associated to $\tau$, to prove Theorem \ref{mainrel} (see \cite[Remark 1.19]{cf:Book}).

In this framework, one can check easily that the sequence of partition functions of the systems of size $N$ at the annealed critical point $h^a(\gb)=-\gl(\gb)$
\begin{equation}
 Z_N^{\gb,-\gl(\gb),\go}=\bE\left[\exp\left(\sum_{n=1}^N (\gb \go_n-\gl(\gb))\ind_{\{n\in \tau\}}\right)\right], \quad  N\in \N,
\end{equation}
is a martingale with respect to the filtration $(\mathcal F_N)_{N\in \N}$, where
\begin{equation}
\mathcal F_N:= \sigma( \go_n, n\le N), 
\end{equation}
is the sigma-algebra generated by the environment seen up to the step $N$.
Since it is non-negative, it converges almost-surely to a limit
\begin{equation}
 Z^{\gb,\go}_{\infty}:=\lim_{N\to \infty} Z_{N}^{\gb,-\gl(\gb),\go}.
\end{equation}
The reader can check that the event $\{Z_{\infty}=0\}$ belongs to the tail sigma algebra $
 \bigcap_{N\in \N}\sigma(\go_n,n\ge N)$,
and hence, has probability $0$ or $1$.

The following result indicates that when the martingale in non-degenerate, disorder does not affect the behavior of $\tau$ at the annealed critical point. It is proved at the end of the paper in  Section \ref{wdivl} where we study the infinite volume limit of the polymer measure.
\begin{proposition}\label{technos}
 Let $\gb>0$ and $\tau$ a recurrent renewal be such that
\begin{equation}
  Z^{\gb,\go}_{\infty}>0 \quad  \bbP \ a.s.
\end{equation}
Then, one has for all $\gamma<\alpha$.
\begin{equation}\label{tructruc}
 \lim_{N\to \infty} \bbE\left[\bP_N^{\gb,-\gl(\gb)}\left( |\tau\cap[1,N]|>N^{\gamma}\right)\right]=1.
\end{equation}
\end{proposition}

The proof of Theorem \ref{mainrel} can be achieved once we have the following criterion for the convergence of the martingale,

\begin{proposition}\label{martinG}
Let $\tau$ be a recurrent renewal. Let $\tau^{(1)}$ and $\tau^{(2)}$ denote two independent copies of $\tau$.
If the renewal process $\tau':=\tau^{(1)}\cap\tau^{(2)}$ is transient, then $Z^{\gb,\go}_{\infty}>0$ $\bbP-a.s.$ for all $\gb<\gb_2$ where
\begin{equation}
 \gb_2:=\inf\left\{\gb \ | \ \gl(2\gb)-2\gl(\gb)>-\log \bP^{\otimes 2} (\tau'_1<\infty)\right\}.
\end{equation}
\end{proposition}

\begin{rem}\rm Using the techniques developed in \cite{cf:Hubert} for hierarchical pinning model with site disorder we could also prove that when the renewal process $\tau'$ is recurrent, then the martingale limit $Z^{\gb,\go}_{\infty}=0$, $\bbP$-$a.s.$ for all $\gb>0$.
\end{rem}

\begin{proof}
It is sufficient to prove that when $\gb<\gb_2$, the martingale $Z_N^{\gb,-\gl(\gb),\go}$ is uniformly integrable
(then $\bbE\left[Z^{\gb,\go}_\infty\right]=1$ so that the limit cannot be uniformly equal to zero). Here we prove the stronger statement that $Z_N^{\gb,-\gl(\gb),\go}$ is bounded in $\bbL^2$.
To compute the second moment, one just has to make use of Fubini's theorem
\begin{multline}
 \bbE\left[\left(Z_N^{\gb,-\gl(\gb),\go}\right)^2\right]=\bE^{\otimes 2} \left[ \bbE\left[ \exp\left(\sum_{n=1}^N \left[\gb\go_n-\gl(\gb)\right](\ind_{\{n\in \tau ^{(1)}\}}+\ind_{\{n\in \tau ^{(2)}\}})\right)\right]\right]\\
=\bE^{\otimes 2}\left[\exp\left(\sum_{n=1}^N \left[\gl(2\gb)-2\gl(\gb)\right]\ind_{\{n\in \tau ^{(1)}\cap \tau ^{(2)} \}}\right)\right],
\end{multline}
where $\bP^{\otimes 2}$ is the probability law of $\tau ^{(1)}$, $\tau ^{(2)}$ which are two independent copies of $\tau$.
The sequence is bounded from above if and only if 
\begin{equation}\label{moinque}
 \bE^{\otimes 2}\left[\exp\left(\sum_{n=1}^\infty \left[\gl(2\gb)-2\gl(\gb)\right]\ind_{\{n\in \tau ^{(1)}\cap \tau ^{(2)} \}}\right)\right]<\infty.
\end{equation}
The quantity $\sum_{n=1}^\infty \ind_{\{n\in \tau ^{(1)}\cap \tau ^{(2)} \}}$ is the total number of return to zero of the renewal $\tau'=\tau^{(1)}\cap \tau^{(2)}$. It is therefore a geometric random variable. Therefore, \eqref{moinque} holds if
\begin{equation}
\gl(2\gb)-2\gl(\gb)< -\log \bP(\tau'_1<\infty ).
\end{equation}
\end{proof}

\begin{rem}\rm The idea of using martingale techniques to prove convergence of the partition function has been inspired by techniques developed by Bolthausen \cite{cf:B} for directed polymer, which have been refined since by numerous authors, including Comets and Yoshida \cite{cf:CY} to describe property of the weak disorder phase.
\end{rem}

\begin{proof}[Proof of Theorem \ref{mainrel} when $\tau$ is recurrent]
First, we show that under the given conditions, $\tau'$ defined above is transient. We compute the expectation of the number of renewal points
\begin{equation}
 \bE[|\tau'\setminus\{0\}|]=\sum_{n=1}^\infty\bP^{\otimes2}[n\in\tau']=\sum_{n=1}^\infty\bP[n\in\tau]^2
\end{equation}
 In \cite{cf:Don}, it is proved that for $\alpha\in(0,1)$,

\begin{equation}
 \bP[n\in\tau]=\frac{\alpha\sin(\pi\alpha)}{\pi L(n)n^{1-\alpha}}(1+o(1)).
\end{equation}
Therefore, if either $\alpha<1/2$ or $\alpha=1/2$ and \eqref{convers} holds,  $\bE[|\tau'\setminus\{0\}|]<\infty$ and the process is transient.
We use Proposition  \ref{technos}  and \ref{martinG}and get that, for all $\gb<\gb_2$, \eqref{tructruc} holds for any $\gamma<\alpha$. Then we use the first part of Proposition \ref{expos} to get 
\begin{equation}
\liminf_{h\to -\gl(\gb)^+} \tf(\gb,h)(h+\gl(\gb))^{-\theta}=\infty.
\end{equation}
 for any $\theta>\alpha^{-1}$.
\end{proof}

\section{Proofs}

\subsection{Proof of Proposition \ref{expos}}

For this proof, one has to introduce the partition function of the system with the end point constrained to be pinned (for notational convenience dependence in $\gb,h,\go$ is omitted when no confusion is possible):
\begin{equation}
 Z_N^c:=\bE\left[\exp\left(\sum_{n=1}^N (\gb\go_n+h)\ind_{\{n\in \tau\}}\right)\ind_{\{N\in \tau\}}\right].
\end{equation}
This partition function can be compared to $Z_N$ with the following inequalities, for any $\alpha_+>\alpha$ 
\begin{equation}\label{compare0}
 Z_N^c\le Z_N \le \left[1+ C N^{1+\alpha_+} \exp(-\gb\go_n-h)\right]Z_N^c.
\end{equation}
where  $C$ is a constant depending only on the law of the renewal and $\alpha_+$ (see the proof \cite[Lemma 4.4]{cf:Book}).
Therefore
\begin{equation}
 \lim_{N\to\infty} \frac{1}{N}\bbE \log  Z_N^c=\tf(\gb,h).
\end{equation}
As $\bbE \log  Z_N^c$ is a super-additive sequence, one has $\tf(\gb,h)\ge N^{-1} \bbE \log  Z_N^c$ for every $N$.

We write $u=h-h_0$. Our aim is to prove that for any $\gep>0$, for $u$ sufficiently small $\tf(\gb,h_0+u)\ge u^{(\gamma-\gep)^{-1}}$.
We fix some $\gep>0$, and define $N=N_u:= \lfloor u^{-(\gamma-\gep)^{-1}} \rfloor$.
Suppose that $u$ is small and such that
\begin{equation}\label{fghjk}
 \bbP\left\{ \bP_N^{\gb,h_0,\go}\left(|\tau\cap[1,N]|> N^\gamma\right)> c \right\}\ge c/2.
\end{equation}

From the definition of the pinning measure we have
\begin{multline} 
 \frac{Z_N^{\gb, h_0+u, \go}}{Z_N^{\gb,h_0,\go}}=\bE_N^{\gb,h_0,\go}\left[\exp\left(u\sum_{n=1}^N \gd_n\right)\right]\\
\ge \max\left(  \exp(N^{\gamma} u)\bP_N^{\gb,h_0,\go}\left\{|\tau\cap[1,N]|\ge N^\gamma\right\}\ ,\ 1\right).
\end{multline}
Taking $\log$ and expectation one both sides one gets 
\begin{multline} \label{compare}
 \bbE \log Z_N^{\gb, h_0+u, \go} \ge \bbE \log Z_N^{\gb,h_0,\go}
+\bbE\left[ \left(N^{\gamma} u+\log \bP_N^{\gb,h_0,\go}\left\{|\tau\cap[1,N]|\ge N^\gamma\right\}\right)_+\right]\\
\ge \log \bP(\tau_1>N)+\frac{c}{2}\left(uN^{\gamma}+\log c\right).
\end{multline}
where $(x)_+=\max(x,0)$ denotes the positive part, and $(x)_-=-\min(x,0)$ denotes the negative part. To get the second line we only used \eqref{fghjk} and bounded $Z^{\gb,h_0,\go}_N$ from below by the probability of having no contacts in $[0,N]$.
We want to bound the constrained partition function so we use \eqref{compare0}
\begin{multline}
 \bbE \log (Z_N^c)^{\gb, h_0+u, \go}\ge \bbE \log Z_N^{\gb, h_0+u, \go}-\bbE \left[\log \left(1+CN^{1+\alpha_+}\exp(-\gb\go_n-h)\right)\right]\\
\ge \bbE \log Z_N^{\gb, h_0+u, \go}-(1+\alpha_+)\log N - C',
\end{multline}
 where $C'$ is a constant that can be chosen uniform in $h\ge h_0$, and that depends on $\gb$ and $C$.
Altogether, using \eqref{compare} and $\bP(\tau_1>N)\ge N^{-\alpha_+}$ when $N$ is large enough (cf. \eqref{regvar}), this gives us
\begin{equation}
\bbE \left[\log(Z_N^c)^{\gb, h_0+u, \go}\right] \ge -(1+2\alpha_+)\log N -C''+\frac{c}{2}uN^{\gamma}.
\end{equation}
Now, using the fact that $u\ge N^{-\gamma+ \gep}/2$, one sees that for $N$ large enough ({\it i.e.} $u$ small enough), 
$\bbE \left[\log (Z_N^c)^{\gb, h_0+u, \go} \right] \ge 1$,
so that
\begin{equation}\label{uiop}
 \tf(\gb,h_0+u)\ge \frac{1}{N}\bbE \left[\log (Z_N^c)^{\gb, h_0+u, \go} \right]\ge N_u^{-1}\ge u^{(\gamma-\gep)^{-1}}.
\end{equation}
To finish the proof, we notice that under condition \eqref{condinf}, \eqref{fghjk} (and therefore \eqref{uiop}) holds for all small $u$, and that under condition \eqref{chemoss}, it holds for a sequence of values of $u$ that tends to zero.
\qed

\subsection{Weak disorder, the infinite volume limit}\label{wdivl}

It is shown in \cite[Chapter 7]{cf:Book} that the limiting polymer measure $\bP_{\infty}^{\gb,h,\go}=\lim_{N\to\infty} \bP_{N}^{\gb,h,\go}$ exists in a weak sense. In this section, we propose to describe accurately this measure at the annealed critical point, when the limit of the martingale $Z^{\gb,\go}_\infty$ is non-degenerate.

Let $\theta$ be the shift operator acting on the environment defined by
\begin{equation}
\theta \go:= (\go_{n+1})_{n\in \N}.
\end{equation}

Let $\mathcal G_N$ be the sigma-algebra generated by $\tau\cap[0,N]$.
For any fixed set $\bar \tau= \left\{ \bar \tau_1, \bar \tau_2,\dots, \bar \tau_n\right\}\subset [0,N]$, \ \ $\bar \tau_1<\bar \tau_2< \dots<\bar \tau_n$ define
\begin{equation}
 \bar \bP^{\gb,\go}_{\infty}\left(\tau\cap[0,N]= \bar \tau\right) = \frac{1}{Z^{\gb,\go}_\infty}\prod_{i=1}^n K(\bar \tau_i-\bar \tau_{i-1})\exp\left( \gb\go_{\bar \tau_i}-\gl(\gb)\right)\sum_{j=N+1}^{\infty} K(j-\tau_n)Z^{\gb,\theta^j\go}_{\infty}.
\end{equation}
One can check that this definition is coherent so that $\bar \bP^{\gb,\go}_{\infty}$ defines a probability measure on $\bigvee_{N\in\N} \mathcal G_N$. Moreover we have (and it is straightforward from the definition)

\begin{proposition}
When $\gb$ is such that $Z^{\gb,\go}_{\infty}>0$ $\bbP-a.s.$,
The sequence of measures $\bP_N^{\gb,-\gl(\gb),\go}$ converges weakly to $\bar \bP_{\infty}^{\gb,\go}$, $\bbP$-almost surely.
\end{proposition}


What we want to show is that when $N$ is large, the measure  $\bE_N^{\gb,-\gl(\gb),\go}$ is, in a sense, very close to the annealed measure.
The complete method developed in \cite{cf:CY} could be applied here to prove that $\tau$ has a scaling limit under $\bE_N^{\gb,-\gl(\gb),\go}$ (the regenerative set of an $\alpha$ stable process, just like the annealed model see \cite[Chapter 2]{cf:Book}). We bound ourselves to show that $\bar\bP^{\gb,\go}_{\infty}$ inherits all the almost-sure features of $\bP$. More precisely

\begin{proposition}\label{toutou}
The measure  $\bbP  \bar \bP^{\gb,\go}_{\infty}$ is absolutely continuous with respect to $\bP$.
\end{proposition}
Proposition \ref{toutou} follows from the generalization of Proposition \ref{technos} below (the second equality).

\begin{lemma}\label{macin}
 Let $A_n$ be a sequence of events with $A_n\in\mathcal G_{n}$ such that $\lim\limits_{n\to\infty}\bP(A_n)=0$.
Then
\begin{equation}
 \lim_{n\to\infty}\sup_{N}\bbE\left[ \bP_N^{\gb,-\gl(\gb),\go}(A_n)\right]=\lim_{n\to\infty}\bbE\left[\bar \bP_\infty^{\gb,\go}(A_n)\right]=0.
\end{equation}
\end{lemma}

\begin{proof}
The proof is very similar to the one of \cite[Lemma 4.2]{cf:CY}. We include it here for the sake of completeness.
We only prove $\lim_{n\to\infty}\sup_{N}\bbE\left[ \bP_N^{\gb,-\gl(\gb),\go}(A_n)\right]$, the other one being similar and simpler.
Let $\gd>0$ be arbitrary
\begin{equation}
 \bbE\left[ \bP_N^{\gb,-\gl(\gb),\go}(A_n)\right]=\bbE\left[\bP_N^{\gb,-\gl(\gb),\go}(A_n)\ind_{Z^{\gb,-\gl(\gb),\go}_N\ge \gd}\right]+ \bbP\left(Z^{\gb,-\gl(\gb),\go}_N< \gd\right).
\end{equation}
The first term on the right hand side can be bounded from above by 
\begin{equation}
 \gd^{-1}\bbE\left[Z^{\gb,-\gl(\gb),\go}_N \bP_N^{\gb,-\gl(\gb),\go}(A_n)\right]=\gd^{-1}\bP(A_n),
\end{equation}
which vanishes when $n$ goes large. As for the second-term, since $(Z^{\gb,-\gl(\gb),\go}_N)^{-1}$ converges almost surely, it is tight sequence, and hence
\begin{equation}
 \lim_{\gd\to 0} \sup_{N} \bbP\left(Z^{\gb,-\gl(\gb),\go}_N< \gd\right)=0.
\end{equation}
\end{proof}

\begin{proof}[Proof of Proposition \ref{technos}]
 We just have to use the preceding Lemma with $A_n:=\{|\tau\cap[0,n]|\le n^{\gamma}\}=\{\tau_{\lfloor n^{\gamma}\rfloor+1}>n\}$, $\gamma<\alpha$.
It is a standard computation to prove that $\lim_{n\to\infty}\bP[A_n]\to 0$: from \cite[XI.5 pp.373 and XIII.6 Theorem 2 (b) pp.448]{cf:Feller2} that $\tau_k/a_k$ converges to in law to an $\alpha$-stable distribution, where $a_k$ is such that $kL(a_k)a_k^{-\alpha}\to 1$. $\lim_{n\to\infty}\bP[A_n]\to 0$, follows from $a_k=o(k^{\frac{1}{\gamma}})$.
Therefore 
\begin{equation}
 \lim_{N\to \infty}\bbE \left[\bP_N^{\gb,-\gl(\gb),\go}(A_N)\right]=0.
\end{equation}

\end{proof}

{\bf Acknowledgements}: The author is very much indebted to Giambattista Giacomin for his precious advice while writing this article, and would like to thank  Fabio Toninelli and Quentin Berger for enlightening discussions. The author also acknowledges the support of ANR grant POLINTBIO and ERC grant PTRELSS.

\end{document}